\documentclass[10pt]{amsart}

\usepackage{url}
\usepackage{pdfsync}

\usepackage{textcomp}
\usepackage[latin1]{inputenc}
\usepackage{amsmath,amsfonts,amssymb}
\usepackage{amscd}
\usepackage{graphicx}
\usepackage{color}
\usepackage[all]{xy}
\usepackage{mathrsfs}

\usepackage{mathrsfs}

\DeclareMathAlphabet{\mathpzc}{OT1}{pzc}{m}{it}

\usepackage[all]{xy}
\usepackage{tikz}
\usetikzlibrary{arrows,decorations.pathmorphing,decorations.pathreplacing,positioning,shapes.geometric,shapes.misc,decorations.markings,decorations.fractals,calc,patterns}

\usepackage{graphicx}

\newcommand{\green}{\color{green}}

\numberwithin{equation}{section}

\newcommand{\cT}{\mathcal{T}}
\newcommand{\cS}{\mathcal{S}}
\newcommand{\N}{\mathbb{N}}
\renewcommand{\P}{\mathcal{P}}
\newcommand{\al}{\alpha}
\DeclareMathOperator{\SL}{SL}

\newtheorem{Lemma}{Lemma}[section]
\newtheorem{Theorem}[Lemma]{Theorem}

\newtheorem{Corollary}[Lemma]{Corollary}

\theoremstyle{definition}
\newtheorem{Definition}[Lemma]{Definition}

\newtheorem{Remark}[Lemma]{Remark}

\newtheorem{Example}[Lemma]{Example}

\begin{document}

\setlength{\parindent}{0pt}
\setlength{\parskip}{2pt}
%The default \baselineskip is close to 4.8mm
%\setlength{\baselineskip}{5.3mm}

\title[Frieze pattern determinants]{Generalized frieze
pattern determinants and higher angulations of polygons}

\author{Christine Bessenrodt}
\address{Institut f\"{u}r Algebra, Zahlentheorie und Diskrete
  Mathematik, Fa\-kul\-t\"{a}t f\"{u}r Ma\-the\-ma\-tik und Physik, Leibniz
  Universit\"{a}t Hannover, Welfengarten 1, 30167 Hannover, Germany}
\email{bessen@math.uni-hannover.de}
\urladdr{http://www.iazd.uni-hannover.de/\~{ }bessen}

\author{Thorsten Holm$^1$}
\address{Institut f\"{u}r Algebra, Zahlentheorie und Diskrete
  Mathematik, Fa\-kul\-t\"{a}t f\"{u}r Ma\-the\-ma\-tik und Physik, Leibniz
  Universit\"{a}t Hannover, Welfengarten 1, 30167 Hannover, Germany}
\email{holm@math.uni-hannover.de}
\urladdr{http://www.iazd.uni-hannover.de/\~{ }tholm}

\author{Peter J\o rgensen}
\address{School of Mathematics and Statistics,
Newcastle University, Newcastle upon Tyne NE1 7RU, United Kingdom}
\email{peter.jorgensen@ncl.ac.uk}
\urladdr{http://www.staff.ncl.ac.uk/peter.jorgensen}

\thanks{{\em Acknowledgement. }This work has been carried out in the framework
  of the research priority programme SPP 1388 {\em Darstellungstheorie} of
  the Deutsche Forschungsgemeinschaft (DFG).  We gratefully acknowledge
  financial support through the grant HO 1880/5-1. }
\thanks{ $^1$ Corresponding author.
Email: {\tt holm@math.uni-hannover.de}, Phone: +49 511 762 4484,
Fax: +49 511 762 5490}
\keywords{Determinant, elementary divisor,
frieze pattern, polygon, Smith normal form, symmetric matrix}

\subjclass[2010]{05E99, 13F60, 15A15, 51M20}

\begin{abstract}
Frieze patterns (in the sense of Conway and Coxeter)
are in close connection to triangulations of polygons.
Broline, Crowe and Isaacs have assigned a symmetric
matrix to each polygon triangulation 
and computed the determinant.
In this paper we consider $d$-angulations of polygons
and generalize the combinatorial algorithm
for computing the entries in the 
associated symmetric matrices; we
compute their determinants and the Smith normal forms. 
It turns out that both are independent of the 
particular $d$-angulation,
the determinant is a power of $d-1$,
and the elementary divisors only take values $d-1$ and $1$.
We also show that in the generalized 
frieze patterns obtained
in our setting every adjacent
$2\times 2$-determinant is 0 or 1, and we give a combinatorial
criterion for when they are 1, which in the case $d=3$ gives back the
Conway-Coxeter condition on frieze patterns.
\end{abstract}

\maketitle

%%%%%%%%%%%%%%%%%%%%%%%%%%%%%%%%%%%%%%%%%%%%%%%%%%%%%%%%%%%%
%%%%%%%%%%%%%%%%%%%%%%%%%%%%%%%%%%%%%%%%%%%%%%%%%%%%%%%%%%%%
\section{Introduction}
\label{sec:introduction}

Frieze patterns have been introduced and studied by Conway and
Coxeter \cite{CC1}, \cite{CC2}.
A frieze pattern (of
size $n$) is an array of $n$ bi-infinite rows of positive integers
(arranged as in the example below)
such that the top and bottom rows consist only of 1's and,
most importantly, every set of four adjacent numbers forming a diamond
$$\begin{matrix}
&b& \\ a&&d\\&c&
\end{matrix}$$
satisfies the determinant condition $ad-bc=1$.
An example of such a frieze pattern is given by
$$
\begin{array}{cccccccccccccccccccccccc}
\ldots & 1 & & 1 & & 1 & & 1 & & {\green 1} & & 1 & & 1 & & 1 & & 1 & & 1 & & 1 & \ldots \\
& \ldots & 1 & & 3 & & 1 & & {\green 2} & & {\green 2} & & 1 & & 3 & & 1 & & 2 & & 2 &
\ldots & & \\
\ldots & 1 & & 2 & & 2 & & {\green 1} & & {\green 3} & & {\green 1} & & 2 & & 2 & & 1 & & 3 & & 1 & \ldots \\
& \ldots & 1 & & 1 & & {\green 1} & & {\green 1} & & {\green 1} & & {\green 1} & & 1 & & 1 & & 1 & & 1 &
\ldots & & \\
\end{array}
$$
A crucial feature of frieze patterns is that they are
invariant under a glide reflection. In the above example,
a fundamental domain for the frieze pattern is given by the green
region; the entire pattern is obtained by iteratively performing a glide
reflection to this region.

Frieze patterns can be constructed geometrically via triangulations
of polygons.
For $n\in \mathbb{N}$, let $\mathcal{P}_n$ be a convex $n$-gon,
and consider any triangulation $\mathcal{T}$ of $\mathcal{P}_n$
(necessarily into $n-2$ triangles).
We label the vertices of $\mathcal{P}_n$ by $1,\ldots,n$ in
counterclockwise order; in the sequel $\mathcal{P}_n$ is
always meant to be the convex $n$-gon together with a fixed 
labelling.

For each vertex $i\in \{1,\ldots,n\}$
let $a_i$ be the number of triangles of $\mathcal{T}$
incident to the vertex~$i$. Then the sequence $a_1,\ldots,a_n$,
repeated infinitely often,
gives the second row in a frieze pattern (of size $n-1$).

As an example, consider the case $n=5$ and the following triangulation
of the pentagon
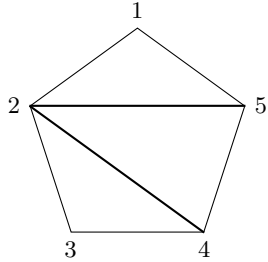
\[
  \begin{tikzpicture}[auto]
    \node[name=s, shape=regular polygon, regular polygon sides=5, minimum size=3cm, draw] {};
    \draw[thick] (s.corner 2) to (s.corner 4);
    \draw[thick] (s.corner 2) to (s.corner 5);
    \draw[shift=(s.corner 1)]  node[above]  {{\small 1}};
  \draw[shift=(s.corner 2)]  node[left]  {{\small 2}};
  \draw[shift=(s.corner 3)]  node[below]  {{\small 3}};
  \draw[shift=(s.corner 4)]  node[below]  {{\small 4}};
  \draw[shift=(s.corner 5)]  node[right]  {{\small 5}};
  \end{tikzpicture}
\]
We get for the number of triangles at the vertices the sequence
$a_1=1$, $a_2=3$, $a_3=1$, $a_4=2$ and $a_5=2$, whose repetition gives
exactly the second row in the above example of a frieze pattern.

A crucial result of Conway and Coxeter is that every frieze
pattern arises in this way from a triangulation.

As a more complicated example, consider the following
triangulation of the octagon
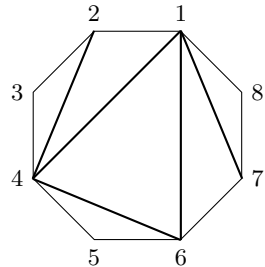
\[
  \begin{tikzpicture}[auto]
    \node[name=s, draw, shape=regular polygon, regular polygon sides=8, minimum size=3cm] {};
    \draw[thick] (s.corner 1) to (s.corner 4);
    \draw[thick] (s.corner 2) to (s.corner 4);
    \draw[thick] (s.corner 4) to (s.corner 6);
    \draw[thick] (s.corner 1) to (s.corner 6);
    \draw[thick] (s.corner 1) to (s.corner 7);
   \draw[shift=(s.corner 1)]  node[above]  {{\small 1}};
  \draw[shift=(s.corner 2)]  node[above]  {{\small 2}};
  \draw[shift=(s.corner 3)]  node[left]  {{\small 3}};
  \draw[shift=(s.corner 4)]  node[left]  {{\small 4}};
  \draw[shift=(s.corner 5)]  node[below]  {{\small 5}};
  \draw[shift=(s.corner 6)]  node[below]  {{\small 6}};
  \draw[shift=(s.corner 7)]  node[right]  {{\small 7}};
  \draw[shift=(s.corner 8)]  node[right]  {{\small 8}};
   \end{tikzpicture}
\]
This triangulation leads to the following frieze pattern
$$
\begin{array}{ccccccccccccccccccccccccccccccccccc}
& 1 & & 1 & & 1 & & {\green 1} & &  {\green 1} & & {\green 1} & & {\green 1}
 & & {\green 1} & & {\green 1} & & {\green 1} & & 1 & & 1 & & 1 & &
 1 & & 1 & & 1
& & 1 & \\
\cdots 2 & & 2 & & 1 & & 4 & & {\green 2} & & {\green 1} & & {\green 4} & &
{\green 1} & & {\green 3} & & {\green 2} & & 1 & & 4 & & 2 & & 1 & & 4 & & 1
& & 3 & & 2 \cdots \\
& 5 & & 1 & & 3 & & 7 & & {\green 1} & & {\green 3} & & {\green 3} & &
{\green 2}
& & {\green 5} & & 1 & & 3 & & 7 & & 1 & & 3 & & 3 & & 2 & & 5 & \\
\cdots 3 & & 2 & & 2 & & 5 & & 3 & & {\green 2} & & {\green 2} & & {\green 5}
& & {\green 3} & & 2 & & 2 & & 5 & & 3 & & 2 & & 2 & & 5 & & 3 & & 2 \cdots \\
& 1 & & 3 & & 3 & & 2 & & 5 & & {\green 1} & & {\green 3} & & {\green 7} &
& 1 & & 3 & & 3 & & 2 & & 5 & & 1 & & 3 & & 7 & & 1 & \\
\cdots 2 & & 1 & & 4 & & 1 & & 3 & & 2 & & {\green 1} & & {\green 4} & & 2 & & 1 & & 4 & & 1
& & 3 & & 2 & & 1 & & 4 & & 2 & & 1\cdots \\
& 1 & & 1 & & 1 & & 1 & & 1 & & 1 & & {\green 1} & & 1 & & 1 & & 1 & & 1 & & 1 & & 1 & & 1 &
& 1 & & 1 & & 1 &  \\
\end{array}
$$

To any triangulation $\cT$ of a convex $n$-gon
$\mathcal{P}_n$, Broline, Crowe and Isaacs \cite{BCI} have
attached a symmetric $n\times n$-matrix associated to the fundamental regions
of the corresponding frieze pattern.
By fixing a labelling of the vertices we choose a particular
fundamental region and thus fix a particular matrix $M_\cT$.
For example, take the green region in the above example;
make it the lower 
triangular part of a symmetric $n\times n$-matrix $M_\cT$ with
0's on the diagonal, and we get the following $8\times 8$-matrix
$$M_\cT=
\begin{pmatrix}
0 & 1 & 2 & 1 & 2 & 1 & 1 & 1 \\
1 & 0 & 1 & 1 & 3 & 2 & 3 & 4 \\
2 & 1 & 0 & 1 & 4 & 3 & 5 & 7 \\
1 & 1 & 1 & 0 & 1 & 1 & 2 & 3 \\
2 & 3 & 4 & 1 & 0 & 1 & 3 & 5 \\
1 & 2 & 3 & 1 & 1 & 0 & 1 & 2 \\
1 & 3 & 5 & 2 & 3 & 1 & 0 & 1 \\
1 & 4 & 7 & 3 & 5 & 2 & 1 & 0 \\
\end{pmatrix}.
$$
The main result of Broline, Crowe and Isaacs gives a simple closed
formula for the determinant of this matrix.
This does not depend on the labelling, i.e., the choice of a fundamental region,
but much more surprisingly, this determinant is independent of the particular
triangulation.

\begin{Theorem}[\cite{BCI}, Theorem 4] \label{thm:BCI}
Let $\cT$ be a triangulation of $\mathcal{P}_n$.
With the above notation we have
$$\det M_\cT = - (-2)^{n-2}\:.$$
\end{Theorem}

Hence in the example above the matrix has determinant
$\det M_\cT = -2^6 = -64$ (which can also be verified directly by an
easy calculation).

In this paper we consider more generally $d$-angulations of 
polygons for arbitrary integers $d\ge 3$. These are also
classic objects in combinatorics which can be traced back 
(at least) to a paper by Cayley \cite{Cayley}. 
More recently, $d$-angulations have for instance appeared prominently 
in the theory of generalized cluster complexes as introduced
by Fomin and Reading \cite{FR}; see also \cite[Section 5.2.3]{Armstrong} 
and the many references therein. 

The main aim of this paper is to generalize the Broline-Crowe-Isaacs
matrices and the results in \cite{BCI}
from triangulations to $d$-angulations for arbitrary
integers $d\ge 3$, and to determine also the Smith normal form
of the corresponding matrices $M_\cT$.

Recall the above simple algorithm for obtaining a matrix
from a triangulation: one just counts the number of triangles
attached to any vertex of the polygon; these numbers give the
second row of a frieze pattern and from the latter one reads off the
matrix $M_{\cT}$. It turns out that this procedure has to be refined
in order to allow a generalization to higher angulations.
The problem is that it no longer suffices to get one line
of a sort of frieze, instead, for a given $d$-angulation $\cT$
of $\mathcal{P}_n$ one has to describe all numbers
in the new matrix $M_\cT$
combinatorially in terms of the $d$-angulation.
This will be achieved using the notion of $d$-paths;
see Definition~\ref{def:sequence}. The resulting matrices
$M_{\cT}$ which generalize the Broline-Crowe-Isaacs matrices
are then shown in Theorem~\ref{thm:seq_matrix}
to be symmetric matrices. For the numbers of $d$-paths
we give in Section~\ref{sec:weights} an alternative description
which simplifies their computation.

As one of the main results in this paper we show in Section
\ref{sec:det}
the following
generalization and refinement
of the above determinant theorem by Broline, Crowe and Isaacs
where we move from triangulations
to arbitrary $d$-angulations and we
also give the elementary divisors (which have not
been determined in \cite{BCI}).

\begin{Theorem}\label{thm:smith}
Let $d\ge 3$ and $m\ge 0$ be integers and set
$n=d+m(d-2)$. Let the matrix $M_\cT$ be associated to a $d$-angulation $\cT$
of $\mathcal{P}_n$.
Then the elementary
divisors of $M_\cT$ are $d-1$ with multiplicity $m+1$ and
the remaining elementary divisors are all~$1$.
The determinant of the matrix is
$$\det M_\cT = (-1)^{n-1} (d-1)^{m+1}.$$
\end{Theorem}

Note that in the case $d=3$ of triangulations of an $n=(m+3)$-gon
we get back the formula
$$\det M_\cT = (-1)^{n-1} 2^{m+1} = (-1)^{n-1} 2^{n-2}
= -(-2)^{n-2}$$
from Theorem~\ref{thm:BCI}. But furthermore,
we also get the refined result
on elementary divisors -- even these are independent of the
particular $d$-angulation of $\mathcal{P}_n$.
\smallskip

In the final Section~\ref{sec:friezes} we study the frieze-like
patterns we obtain from $d$-angulations. As for the 
Conway-Coxeter frieze patterns we use the numbers of $d$-paths
occurring in the lower triangular part of the symmetric
matrices $M_{\cT}$ as a fundamental region from which the entire
pattern is created using glide reflections. We
determine in Theorem~\ref{thm:hinge} the possible values of the
determinant of any adjacent $2\times 2$-matrix in these patterns.
It turns out that they only take values $0$ and $1$ and we give
an explicit combinatorial criterion for which of these determinants
are 0 and which are 1. For the case $d=3$ of triangulations
it is then easy to see that the case of determinant 0 does not occur,
thus reproving the Conway-Coxeter condition that all
determinants should be 1.
\smallskip

In a recent paper, Baur and Marsh \cite{BM} have generalized the
Broline-Crowe-Isaacs result on the determinant in a different way,
namely to frieze patterns arising from Fomin and Zelevinsky's
cluster algebras of Dynkin type
$A$ and they have given a representation-theoretic interpretation
inside the root category of type $A$.
\smallskip

The combinatorial methods by Broline, Crowe and Isaacs for computing
the entries of a frieze pattern from a triangulation have recently
found applications in the context of $\SL_2$-tilings. The latter
have been introduced by Assem, Reutenauer and Smith~\cite{ARS}
in connection with Fomin and Zelevinsky's cluster algebras
as a tool for obtaining closed formulae for cluster variables.
An $\SL_2$-tiling assigns a positive integer to each vertex in the
$\mathbb{Z}\times \mathbb{Z}$-grid in the plane such that every
adjacent $2\times 2$-matrix has determinant~1. Note the similarity
to the condition for frieze patterns but the limiting
top and bottom rows of 1's disappear and one now has a pattern on
the entire plane.
Applying the Broline-Crowe-Isaacs method to triangulations of some
infinite combinatorial object, many new $\SL_2$-tilings have been
found in \cite{HJ} and moreover a geometrical
interpretation for them has been provided
(which is new even for the previously known $\SL_2$-tilings). 
The combinatorial results on $d$-angulations in the present paper have 
recently also lead to new algebraic insights in the context 
of cluster categories.  
A key ingredient in the categorification of cluster algebras via
cluster categories is the Caldero-Chapoton
map \cite{CC} which maps indecomposable objects in the cluster 
category to cluster variables in the corresponding cluster algebra. 
A nice feature of this map is that it gives rise to friezes (in the 
sense of \cite{ARS}) and this in particular covers the Conway-Coxeter
friezes, see \cite[Section 5]{CC}. 
The frieze patterns from $d$-angulations considered in this paper 
are not covered by the classic Caldero-Chapoton map.  
In \cite{HJ-CC} we present a modification of the Caldero-Chapoton map
which now only depends on a rigid object (instead of a cluster
tilting object). This modified map then produces
socalled generalised friezes (including the ones 
of the present paper) and we give an algebraic
explanation for which of the neighbouring $2\times 2$-determinants 
have value 0 or 1, respectively.

%%%%%%%%%%%%%%%%%%%%%%%%%%%%%%%%%%%%%%%%%%%%%%%%%%%%%%%%%%%
%%%%%%%%%%%%%%%%%%%%%%%%%%%%%%%%%%%%%%%%%%%%%%%%%%%%%%%%%%%
\section{Sequences of polygons in higher angulations}

In this section we generalize the basic construction
for obtaining a frieze from a triangulation.
Instead of triangulations we consider $d$-angulations of
convex polygons for an arbitrary integer $d\ge 3$.

It is not hard to show inductively that a convex
$n$-gon $\mathcal{P}_n$
can be divided into $d$-gons if and only if $n$ is of the form
$n=d + m(d-2)$ for some $m\in \mathbb{N}_0$; in this case there are
precisely $m+1$ of the $d$-gons in the $d$-angulation (or, alternatively,
precisely $m$ diagonals).

Recall that the vertices of $\mathcal{P}_n$
are labelled $1,2,\ldots,n$ in counterclockwise order, and
this numbering has to be taken modulo~$n$ below.

\begin{Definition} \label{def:sequence}
Let $n=d+m(d-2)$ for some $d\ge 3$ and $m\in \mathbb{N}_0$.
Let $\cT$ be a $d$-angulation of
$\mathcal{P}_n$, and let $i$ and $j$ be vertices of
$\mathcal{P}_n$.
\begin{itemize}
\item[{(a)}] A sequence $p_{i+1},p_{i+2},\ldots,
p_{j-2},p_{j-1}$ of $d$-gons of $\cT$
is called a (counterclockwise) $d$-path from $i$ to~$j$ if it
satisfies the following properties:
\begin{itemize}
\item[{(i)}] The $d$-gon $p_k$ is incident to vertex~$k$,
for all $k\in \{i+1, i+2,\ldots,j-1\}$.
\item[{(ii)}] In the sequence $p_{i+1},p_{i+2},
\ldots, p_{j-2},p_{j-1}$, every $d$-gon of $\cT$
appears at most $d-2$ times.
\end{itemize}
\item[{(b)}] The number of (counterclockwise)
$d$-paths from $i$ to $j$
is denoted $m_{i,j}$.
\end{itemize}
We then define the $n\times n$-matrix
$M_\cT$ associated to the $d$-angulation $\cT$
by $M_\cT=(m_{i,j})_{1\leq i,j\leq n}$.
\end{Definition}

Roughly speaking, for a $d$-path we go counterclockwise from $i$ to $j$ and at each
intermediate vertex we pick an attached $d$-gon, so that in total
every $d$-gon appears at most $d-2$ times in the sequence.

\begin{Remark} \label{rem:d-path}
When $n=d+m(d-2)$, the polygon  $\mathcal{P}_n$ is dissected into $m+1$ $d$-gons
and any $d$-path is a sequence of length at most $n-2=(m+1)(d-2)$,
i.e., a $d$-path can never go around $\mathcal{P}_n$
full circle (or more).
In particular, for the numbers $m_{i,j}$ associated
to a $d$-angulation of $\mathcal{P}_n$
we have the properties:
$m_{i,i}=0$ and $m_{i,i+1}=1$,
where the only $d$-path from $i$ to $i+1$ is the empty sequence. 
On the other hand, any $d$-path from $i+1$ to $i$ is of length $n-2$, hence it has to contain every $d$-gon of $\mathcal{T}$ exactly $d-2$ times. 
\end{Remark}

Note that the definition above
provides a direct generalization of the construction
in \cite[p.173]{BCI} for triangulations (where each
triangle was allowed to appear at most once).

Before studying the above sequences and numbers in more detail
let us illustrate Definition~\ref{def:sequence} with an example.

\begin{Example} 
Let $d=4$ and $m=3$, thus $n=d+m(d-2)=10$, and we consider the
following quadrangulation $\cT$ of the $10$-gon.
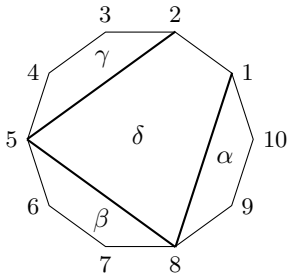
\[
  \begin{tikzpicture}[auto]
    \node[name=s, shape=regular polygon, regular polygon sides=10, minimum size=3cm, draw] {};
 \draw[thick] (s.corner 1) to node[near start, left=7pt] {$\gamma$} (s.corner 4);
 \draw[thick] (s.corner 4) to node[near end, above=25pt] {$\delta$}
 node[below=2pt] {$\beta$} (s.corner 7);
 \draw[thick] (s.corner 7) to node[very near end, below=18pt] {$\alpha$}
 (s.corner 10);
  \draw[shift=(s.corner 1)]  node[above]  {{\small 2}};
  \draw[shift=(s.corner 2)]  node[above]  {{\small 3}};
  \draw[shift=(s.corner 3)]  node[left]  {{\small 4}};
  \draw[shift=(s.corner 4)]  node[left]  {{\small 5}};
  \draw[shift=(s.corner 5)]  node[left]  {{\small 6}};
  \draw[shift=(s.corner 6)]  node[below]  {{\small 7}};
  \draw[shift=(s.corner 7)]  node[below]  {{\small 8}};
  \draw[shift=(s.corner 8)]  node[right]  {{\small 9}};
  \draw[shift=(s.corner 9)]  node[right]  {{\small 10}};
  \draw[shift=(s.corner 10)]  node[right]  {{\small 1}};
  \end{tikzpicture}
\]
We shall compute some of the numbers $m_{i,j}$ and list the
corresponding sequences of quadrangles explicitly. By Definition~\ref{def:sequence}, no quadrangle is allowed to appear more than
twice in a 4-path.

Let $i=2$ and $j=6$. Note that at vertices $3$ and $4$ one has to
choose $\gamma$, and one can choose $\delta$ or $\beta$ at vertex 5.
So the only 4-paths from vertex~2 to vertex~6 are
$\gamma,\gamma,\delta$ and
$\gamma,\gamma,\beta$, and hence $m_{2,6}=2$.

Consider now $i=4$ and $j=9$. Note that at both vertices $6$ and $7$
we have to choose $\beta$. With this already twofold appearance,
$\beta$ must not be chosen at any other vertex. This leaves
$\gamma$ or $\delta$ for vertex 5, and $\alpha$ or $\delta$ for
vertex 8.  So we get four 4-paths from vertex~4 to vertex~9:
$\gamma,\beta,\beta,\alpha$;
$\gamma,\beta,\beta,\delta$;
$\delta,\beta,\beta,\alpha$
and
$\delta,\beta,\beta,\delta$,
and hence $m_{4,9}=4$.

The entire matrix $M_\cT$ can be computed to have the following 
form,
$$M_\cT = 
\begin{pmatrix}
0 & 1 & 2 & 2 & 1 & 2 & 2 & 1 & 1 & 1 \\
1 & 0 & 1 & 1 & 1 & 2 & 2 & 1 & 2 & 2  \\
2 & 1 & 0 & 1 & 1 & 3 & 3 & 2 & 4 & 4  \\
2 & 1 & 1 & 0 & 1 & 3 & 3 & 2 & 4 & 4  \\
1 & 1 & 1 & 1 & 0 & 1 & 1 & 1 & 2 & 2  \\
2 & 2 & 3 & 3 & 1 & 0 & 1 & 1 & 3 & 3  \\
2 & 2 & 3 & 3 & 1 & 1 & 0 & 1 & 3 & 3  \\
1 & 1 & 2 & 2 & 1 & 1 & 1 & 0 & 1 & 1  \\
1 & 2 & 4 & 4 & 2 & 3 & 3 & 1 & 0 & 1  \\
1 & 2 & 4 & 4 & 2 & 3 & 3 & 1 & 1 & 0  \\
\end{pmatrix}.
$$
\end{Example}

We shall use the following notion frequently below. A $d$-gon $\alpha$
of a $d$-angulation $\mathcal{T}$ of a convex polygon $\mathcal{P}$
is called a {\em boundary $d$-gon} if at most
one of the boundary edges
of $\alpha$ is a diagonal in the interior of $\mathcal{P}$ (and the other
are on the boundary of $\mathcal{P}$). 
Note that the case with no
boundary edge of $\alpha$ in the interior of $\mathcal{P}$ only occurs 
when $\alpha=\mathcal{P}$ itself is a $d$-gon. When precisely one
boundary edge of $\alpha$ is in the interior of $\mathcal{P}$ then,
loosely speaking, $\alpha$ is a
$d$-gon of $\mathcal{T}$ which is cut off by one diagonal of $\mathcal{P}$.
It is easy to see by induction that such a boundary $d$-gon exists
for every $d$-angulation and that every $d$-angulation with at least
one diagonal contains at least two boundary $d$-gons.

\begin{Theorem} \label{thm:seq_matrix}
Let $n=d+m(d-2)$, with $d\ge 3$ and
$m\in \N_0$. Let $\cT$ be a $d$-angulation of
$\mathcal{P}_n$, and let $M_\cT=(m_{i,j})$ be its associated
matrix of numbers of $d$-paths. Then
$m_{i,j}=m_{j,i}$ for all $1\leq i,j\leq n$, i.e.,
$M_\cT$ is a symmetric matrix.
\end{Theorem}

\begin{proof}
We prove the result by induction on~$m$.
For $m=0$, we have just a $d$-gon $\al=\P_d$ (with no diagonals),
and we clearly have $m_{i,j}=1$ for all $i\neq j$,
the only $d$-path from $i$ to $j$ being
$p_{i+1}=\al,\al,\ldots,p_{j-1}=\al$.

Now assume that we have already proved the claim for $d$-angulations of
$\mathcal{P}_n$.
We consider a $d$-angulation~$\cT$ of $\P_{n+d-2}$ and want to prove
symmetry for the matrix $M_\cT=(m_{i,j})_{1\le i,j\le n+d-2}$.
Let $\al$ be a boundary $d$-gon of $\mathcal{T}$ cut off by the
diagonal~$t$;
w.l.o.g.\ we may assume that this is a diagonal between the vertices 
$1$ and $n$,
so that $\al$ has vertices $1,n,n+1,\ldots, n+d-2$, as in the
following figure. 
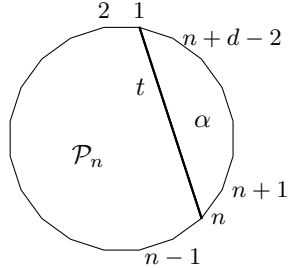
\[
  \begin{tikzpicture}[auto]
    \node[name=s, shape=regular polygon, regular polygon sides=20, minimum size=3cm, draw] {};
 \draw[thick] (s.corner 1) to (s.corner 14);
 \draw[thick] (s.corner 1) to (s.corner 14);
  \draw[shift=(s.corner 17)]  node[left=5pt]  {$\alpha$};
  \draw[shift=(s.corner 5)]  node[right=40pt]  {$t$};
  \draw[shift=(s.corner 7)]  node[right=20pt]  {$\mathcal{P}_n$};
  \draw[shift=(s.corner 1)]  node[above]  {{\small 1}};
  \draw[shift=(s.corner 2)]  node[above]  {{\small 2}};
  \draw[shift=(s.corner 14)]  node[right]  {{\small $n$}};
  \draw[shift=(s.corner 15)]  node[right]  {{\small $n+1$}};
  \draw[shift=(s.corner 20)]  node[right]  {{\small $n+d-2$}};
  \draw[shift=(s.corner 13)]  node[below]  {{\small $n-1$}};
  \end{tikzpicture}
\]
We denote by $\P_n$ the $n$-gon with vertices $1,2,\ldots,n$
obtained from $\P_{n+d-2}$ by cutting off~$\al$;
let $\cT'$ be the $d$-angulation of $\P_n$ obtained from $\cT$
by restriction, and let $M'=M_{\cT'}=(m_{i,j}')_{1\le i,j\le n}$ be the corresponding
matrix of $d$-path numbers.
We now compare the path numbers. We already know
from Remark~\ref{rem:d-path} 
that $m_{i,i}=0$ and
$m_{i,i}'=0$ for all vertices $i$ in the respective $d$-angulated polygons.

{\em Case~1:}
Let $i,j \in \{1,\ldots,n\}$ with $i\le j$.
We clearly have $m_{i,j}=m_{i,j}'$, as the $d$-paths counted are the same
in this case.
On the other hand, also $m_{j,i}=m_{j,i}'$, since there is a bijection
mapping $d$-paths from $j$ to $i$ in $\P_n$ to $d$-paths from $j$ to $i$
in $\P_{n+d-2}$ given by inserting the $d$-gon $\al$ with multiplicity
$d-2$:
$$p_{j+1},\ldots,p_n,p_1,\ldots,p_{i-1}
\mapsto  p_{j+1},\ldots,p_n,p_{n+1}=\al, \al,\ldots, \al, p_{n+d-2}=\al,  p_1,\ldots,p_{i-1}\:. $$
Note here that the vertices $n+1,\ldots,n+d-2$ are only incident to the $d$-gon $\al$ in
the $d$-angulation $\cT$.

{\em Case~2:}
Next we consider two vertices $i\in \{1,\ldots,n\}$ and
$j\in \{n+1,\ldots,n+d-2\}$.
We claim that $m_{i,j}=m_{i,n}'+m_{i,1}'$.
For this, note that a $d$-path from $i$ to $j$ in $\P_{n+d-2}$
has the form
$$p_{i+1}, \ldots, p_n, p_{n+1}=\al, \al, \ldots, p_{j-1}=\al,
$$
and we distinguish the cases $p_n=\al$ and $p_n\ne \al$.
The $d$-paths with $p_n=\al$ correspond bijectively to $d$-paths
$p_{i+1}, \ldots, p_{n-1}$ from $i$ to $n$ in $\P_n$,
the ones with $p_n\ne \al$ correspond bijectively
to $d$-paths $p_{i+1}, \ldots, p_n$ from $i$ to $1$ in $\P_n$.
Similarly, $m_{j,i}=m_{1,i}'+m_{n,i}'$.
Here $d$-paths from $j$ to $i$ in $\P_{n+d-2}$
have the form $p_{j+1}=\al,\al \ldots, p_{n+d-2}=\al,p_1,\ldots, p_{i-1}$,
and we distinguish the cases $p_1=\al$ and $p_1\ne \al$.
The ones with $p_1=\al$ correspond bijectively to $d$-paths
$p_2, \ldots, p_{i-1}$ from $1$ to $i$ in $\P_n$,
the ones with $p_1\ne \al$ correspond bijectively
to $d$-paths $p_1, \ldots, p_{i-1}$ from $n$ to $i$ in $\P_n$.
Thus by induction and symmetry of $M'$, we have $m_{i,j}=m_{j,i}$.

{\em Case~3:}
Finally, we consider two vertices $i,j\in \{n+1,\ldots,n+d-2\}$.
When $i<j$, we clearly have $m_{i,j}=1$.
Now consider a $d$-path from $j$ to $i$; this has the form
$$p_{j+1}=\al, \ldots, p_{n+d-2}=\al, p_1, p_2, \ldots p_{n-1},p_n, p_{n+1}=\al, \ldots, p_{i-1}=\al \,.$$
Then the sequence $p_2,\ldots,p_{n-1}$ is a $d$-path
from $1$ to $n$ in $\P_n$; here necessarily each
$d$-gon with respect to the $d$-angulation $\cT'$
appears exactly $d-2$ times (cf. Remark~\ref{rem:d-path}),
hence we must have $p_1=\al=p_n$
in the original sequence.
Furthermore, we have by induction
$m_{1,n}'=m_{n,1}'=1$, i.e.,  the sequence $p_2,\ldots,p_{n-1}$ is unique
and thus also the $d$-path from $j$ to $i$ in $\P_{n+d-2}$,
giving $m_{j,i}=1$.

This completes the proof of the symmetry of the matrix~$M_\cT$.
\end{proof}

\begin{Remark}
In the proof above, we have provided some explicit bijections with 
good properties in the induction steps. Indeed, in these constructions 
$d$-paths from $i$ to $j$ correspond to $d$-paths from $j$ to $i$ where 
a complementary choice is applied relative to picking each $d$-gon
$d-2$ times in total.
In a suitable context (which we refrain from expounding in this paper) 
a precise result towards this complementary symmetry can be formulated 
and proved by a modification of the arguments above. 
One has to be careful, though, about the order of picking the $d$-gons 
along the $d$-paths. This is already illustrated by the correspondence 
of the one empty path from $i$ to $i+1$ and one full round from $i+1$ to 
$i$. Hence an explicit
bijection between the $d$-paths from $i$ to $j$ and $j$ to $i$ has to 
follow the geometry of the $d$-angulation very closely. 
\end{Remark}

In the course of the proof we have shown explicitly how the matrices
associated to a $d$-angulation and its restriction to the polygon obtained
by cutting off one boundary 
$d$-gon are related. For later usage we record this here,
keeping the notation of the induction step of the proof.

\begin{Corollary}\label{cor:recursion}
Let $\cT$, $\cT'$ be  as above, with associated matrices
$M_{\mathcal{T}}=(m_{i,j})_{1\le i,j\le n+d-2}$ and
$M'=(m_{i,j}')_{1\le i,j\le n}$.
Setting $r_i':=m_{i,1}'+m_{i,n}'= m_{1,i}' + m_{n,i}'$
for abbreviation, we then have
$$
M_{\mathcal{T}}=
\left(
\begin{array}{ccccc|cccc}
  & & & & & 1 & \ldots & \ldots & 1 \\
  & & & & & r_2' & \ldots & \ldots & r_2' \\
  && M' && & \vdots & & & \vdots \\
  & & & & & r'_{n-1} & \ldots & \ldots & r'_{n-1} \\
 & & & & & 1 & \ldots & \ldots & 1 \\
 \hline
 1 & r'_2 & \ldots & r'_{n-1} & 1 & 0 & 1 & \ldots & 1\\
 \vdots & \vdots & & \vdots & \vdots & 1 & 0 & \ddots &\vdots\\
 \vdots & \vdots & & \vdots & \vdots & \vdots & \ddots
 & \ddots & 1\\
 1 & r'_2 & \ldots & r'_{n-1} & 1 & 1 & \ldots & 1 & 0\\
\end{array}
\right).
$$
\end{Corollary}

\begin{Remark}
In particular, the rows in the matrix in the top right corner
and the columns in the matrix in the bottom left corner
are constant, as shown in Case~2 of the proof above. 
Moreover, as noted above, for $i=1$ and $i=n$ we have
$m_{1,1}'+m_{1,n}'=m_{1,n}'=1=m_{n,1}'=m_{n,1}'+m_{n,n}'$.
The entries in the top left and the bottom right corner
have been dealt with in Case~1 and Case~3, respectively. 
\end{Remark}

\begin{Remark}
Completely analogous to Definition~\ref{def:sequence} we could have
defined numbers $\bar m_{i,j}$ counting {\em clockwise} $d$-paths from
vertex $i$ to vertex $j$.
It is immediate from the definition that then $\bar m_{i,j}=m_{j,i}$
(just reverse the sequences obtained). Then Theorem~\ref{thm:seq_matrix}
implies that for every pair of vertices
we have $m_{i,j}=\bar m_{i,j}$, i.e., the number of $d$-paths
does not depend on whether we count 
clockwise or counterclockwise $d$-paths 
and we can omit mentioning the direction.
Note that this is not at all obvious from the original definition.
\end{Remark}

%%%%%%%%%%%%%%%%%%%%%%%%%%%%%%%%%%%%%%%%%%%%%%%%%%%%%%%%%%%
%%%%%%%%%%%%%%%%%%%%%%%%%%%%%%%%%%%%%%%%%%%%%%%%%%%%%%%%%%%
\section{An alternative description}
\label{sec:weights}

In this section we shall give an alternative description of
the numbers $m_{i,j}$ of $d$-paths from Definition~\ref{def:sequence}.
We will first present an inductive algorithm performed
on any $d$-angulation of a convex polygon, yielding a non-negative
integer $\widetilde{m}_{i,j}$ for any pair of vertices $i,j$ of the
polygon. As the main result of this section we will then
show that these new numbers agree with the numbers $m_{i,j}$ counting
$d$-paths.

\begin{Definition} \label{def:weights}
Let $n=d + m(d-2)$ for some integers $d\ge 3$ and $m\in \mathbb{N}_0$.
For any $d$-angulation
$\cT$ of~$\mathcal{P}_n$ we define an
$n\times n$-matrix $\widetilde{M}_\cT=(\widetilde{m}_{i,j})$ by the
following combinatorial procedure.

Consider a vertex $i\in \{1,\ldots,n\}$ of $\mathcal{P}_n$. To any vertex
$j$ of $\mathcal{P}_n$
we want to assign inductively
a non-negative integer $\widetilde{m}_{ij}$.
We set $\widetilde{m}_{i,i}=0$, and $\widetilde{m}_{i,j}=1$ if
there is a $d$-gon in $\cT$ containing $i$ and $j$. For an arbitrary
$d$-gon $\al$ in $\cT$ we can inductively assume that two of its
vertices, say $k$ and $\ell$, have already been assigned numbers
$\widetilde{m}_{i,k}$ and $\widetilde{m}_{i,\ell}$, respectively. Then for
any other vertex $v$ in $\al$ we set
$\widetilde{m}_{i,v}= \widetilde{m}_{i,k} + \widetilde{m}_{i,\ell}$.
(One can easily convince oneself that this procedure is well-defined.)
\end{Definition}

\begin{Example} \label{ex:12gon}
Let us consider an example for the case $d=4$ and $m=4$ to illustrate
this combinatorial algorithm. We consider the quadrangulation
of a $12$-gon on the left in the following figure
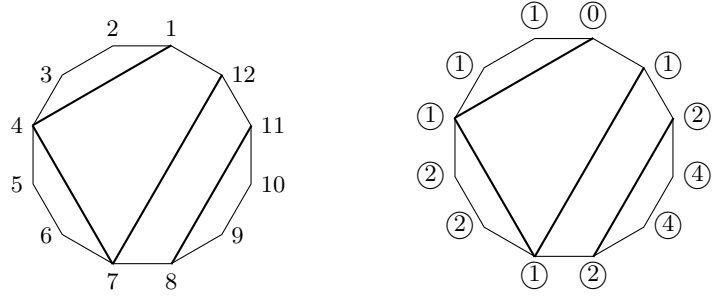
\[
  \begin{tikzpicture}[auto]
    \node[name=s, shape=regular polygon, regular polygon sides=12, minimum size=3cm, draw] {};
    \draw[thick] (s.corner 1) to (s.corner 4);
    \draw[thick] (s.corner 4) to (s.corner 7);
    \draw[thick] (s.corner 7) to (s.corner 12);
    \draw[thick] (s.corner 8) to (s.corner 11);
  \draw[shift=(s.corner 1)]  node[above]  {{\small 1}};
  \draw[shift=(s.corner 2)]  node[above]  {{\small 2}};
  \draw[shift=(s.corner 3)]  node[left]  {{\small 3}};
  \draw[shift=(s.corner 4)]  node[left]  {{\small 4}};
  \draw[shift=(s.corner 5)]  node[left]  {{\small 5}};
  \draw[shift=(s.corner 6)]  node[left]  {{\small 6}};
  \draw[shift=(s.corner 7)]  node[below]  {{\small 7}};
  \draw[shift=(s.corner 8)]  node[below]  {{\small 8}};
  \draw[shift=(s.corner 9)]  node[right]  {{\small 9}};
  \draw[shift=(s.corner 10)]  node[right]  {{\small 10}};
  \draw[shift=(s.corner 11)]  node[right]  {{\small 11}};
  \draw[shift=(s.corner 12)]  node[right]  {{\small 12}};
  \end{tikzpicture}
\hskip1.5cm
\begin{tikzpicture}[auto]
    \node[name=s, shape=regular polygon, regular polygon sides=12, minimum size=3cm, draw] {};
    \draw[thick] (s.corner 1) to (s.corner 4);
    \draw[thick] (s.corner 4) to (s.corner 7);
    \draw[thick] (s.corner 7) to (s.corner 12);
    \draw[thick] (s.corner 8) to (s.corner 11);
  \draw[shift=(s.corner 1)]  node[above]  {\textcircled{{\small 0}}};
  \draw[shift=(s.corner 2)]  node[above]  {\textcircled{{\small 1}}};
  \draw[shift=(s.corner 3)]  node[left]  {\textcircled{{\small 1}}};
  \draw[shift=(s.corner 4)]  node[left]  {\textcircled{{\small 1}}};
  \draw[shift=(s.corner 5)]  node[left]  {\textcircled{{\small 2}}};
  \draw[shift=(s.corner 6)]  node[left]  {\textcircled{{\small 2}}};
  \draw[shift=(s.corner 7)]  node[below]  {\textcircled{{\small 1}}};
  \draw[shift=(s.corner 8)]  node[below]  {\textcircled{{\small 2}}};
  \draw[shift=(s.corner 9)]  node[right]  {\textcircled{{\small 4}}};
  \draw[shift=(s.corner 10)]  node[right]  {\textcircled{{\small 4}}};
  \draw[shift=(s.corner 11)]  node[right]  {\textcircled{{\small 2}}};
  \draw[shift=(s.corner 12)]  node[right]  {\textcircled{{\small 1}}};
  \end{tikzpicture}
\]
We compute the numbers $\widetilde{m}_{1,j}$ which are
indicated by the encircled numbers on the right in the figure.
In the first step we set $\widetilde{m}_{1,1}=0$ and
$\widetilde{m}_{1,2}=\widetilde{m}_{1,3}=
\widetilde{m}_{1,4}=\widetilde{m}_{1,7}=\widetilde{m}_{1,12}=1$
since these are the vertices
appearing in a common quadrangle with vertex 1. In the next step
we set $\widetilde{m}_{1,5}=\widetilde{m}_{1,6}=
\widetilde{m}_{1,4}+\widetilde{m}_{1,7} = 1+1=2$ (for the quadrangle
with vertices $4,5,6,7$), and
$\widetilde{m}_{1,8}=\widetilde{m}_{1,11}=
\widetilde{m}_{1,7}+\widetilde{m}_{1,12} = 1+1=2$ (for the quadrangle
with vertices $7,8,11,12$).
Finally, the remaining vertices $9$ and $10$ get assigned
$\widetilde{m}_{1,9}=\widetilde{m}_{1,10}=
\widetilde{m}_{1,8} + \widetilde{m}_{1,11} = 2 + 2 = 4$.

In this way we obtain the following matrix (we leave the computational
details to the reader),
$$\widetilde{M}_\cT=
\left(
\begin{array}{cccccccccccc}
0 & 1 & 1 & 1 & 2 & 2 & 1 & 2 & 4 & 4 & 2 & 1 \\
1 & 0 & 1 & 1 & 3 & 3 & 2 & 4 & 8 & 8 & 4 & 2 \\
1 & 1 & 0 & 1 & 3 & 3 & 2 & 4 & 8 & 8 & 4 & 2 \\
1 & 1 & 1 & 0 & 1 & 1 & 1 & 2 & 4 & 4 & 2 & 1 \\
2 & 3 & 3 & 1 & 0 & 1 & 1 & 3 & 6 & 6 & 3 & 2 \\
2 & 3 & 3 & 1 & 1 & 0 & 1 & 3 & 6 & 6 & 3 & 2 \\
1 & 2 & 2 & 1 & 1 & 1 & 0 & 1 & 2 & 2 & 1 & 1 \\
2 & 4 & 4 & 2 & 3 & 3 & 1 & 0 & 1 & 1 & 1 & 1 \\
4 & 8 & 8 & 4 & 6 & 6 & 2 & 1 & 0 & 1 & 1 & 2 \\
4 & 8 & 8 & 4 & 6 & 6 & 2 & 1 & 1 & 0 & 1 & 2 \\
2 & 4 & 4 & 2 & 3 & 3 & 1 & 1 & 1 & 1 & 0 & 1 \\
1 & 2 & 2 & 1 & 2 & 2 & 1 & 1 & 2 & 2 & 1 & 0 \\
\end{array}
\right).
$$
\end{Example}

We now come to the main result of this section. 

\begin{Theorem} \label{thm:weights}
Let $n=d+m(d-2)$, with integers $d\ge 3$ and $m\in \N_0$, and
let $\cT$ be a $d$-angulation of $\mathcal{P}_n$. Then
we have $\widetilde{M}_\cT = M_\cT .$
In particular, the matrix $\widetilde{M}_\cT$ is symmetric.
\end{Theorem}

\begin{proof}
Once the equality $\widetilde{M}_\cT = M_\cT$ has been proven the
symmetry of $\widetilde{M}_\cT$ follows from
Theorem~\ref{thm:seq_matrix}.

So we have to show that for any vertices $i,j$ the numbers $m_{i,j}$
from Definition~\ref{def:sequence} and $\widetilde{m}_{i,j}$ from
Definition~\ref{def:weights} coincide.

If $i=j$ then $m_{i,i}=0=\widetilde{m}_{i,i}$ by
Remark~\ref{rem:d-path} and Definition~\ref{def:weights}, respectively.

Now let $i$ and $j$ be different vertices which belong to
a common $d$-gon $\beta$
of $\mathcal{T}$. Then
$\widetilde{m}_{i,j}=1$ by Definition~\ref{def:weights}. So we have to show that
also the number $m_{i,j}$ of $d$-paths is equal to 1.

Recall Case~1 of the proof of Theorem~\ref{thm:seq_matrix};
there it has been
shown that if $\mathcal{P}_n$ is divided as
$\mathcal{P}_n=\mathcal{P}' \cup \alpha$ where $\alpha$ is a boundary
$d$-gon of $\mathcal{T}$, then for any vertices $i,j$ of $\mathcal{P}'$
we have $m_{i,j}=m'_{i,j}$, i.e., the numbers of $d$-paths can be
computed entirely within the smaller polygon $\mathcal{P}'$.

In our situation the vertices $i$ and $j$ belong to a common $d$-gon
$\beta$ of $\mathcal{T}$. From $\mathcal{P}_n$ we
can hence successively remove boundary $d$-gons until only $\beta$
is left, and the proof of Case~1 of Theorem~\ref{thm:seq_matrix}
tells us that in each removal step the number 
of $d$-paths from $i$ to $j$ is not changed.
Thus, it suffices to show that $m_{i,j}=1$ for a single $d$-gon $\beta$;
but this is obvious since the only possible $d$-path is
$p_{i+1}=\beta,\beta,\ldots,\beta=p_{j-1}$.

Finally, we consider the case that $i$ and $j$ are not contained
in a common $d$-gon of $\mathcal{T}$.
By Definition~\ref{def:weights} we have
$\widetilde{m}_{i,j} = \widetilde{m}_{i,k} + \widetilde{m}_{i,\ell}$
where $k$ and $\ell$ are two vertices in a common $d$-gon
$\alpha$ with $j$ where inductively the numbers
$\widetilde{m}_{i,k}$ and $\widetilde{m}_{i,\ell}$ have already been
defined.

Note that by the same argument as used above, we may assume that 
$\al$ is a boundary $d$-gon;  
but then we are exactly in the situation dealt with in Case~2 of the 
proof of Theorem~\ref{thm:seq_matrix}.
There it has been shown that in this
situation we have
$m_{i,j} = m'_{i,k} + m'_{i,\ell}$
where $m'_{i,k}$ and $m'_{i,\ell}$ are the numbers of $d$-paths
computed in the smaller polygon (without $\alpha$).
Inductively, we can assume that $m'_{i,k}=\widetilde{m}_{i,k}$
and $m'_{i,\ell}=\widetilde{m}_{i,\ell}$. Putting everything together
we conclude that
$m_{i,j} = \widetilde{m}_{i,k} + \widetilde{m}_{i,\ell} =
\widetilde{m}_{i,j}$.

This completes the proof of Theorem~\ref{thm:weights}.
\end{proof}

%%%%%%%%%%%%%%%%%%%%%%%%%%%%%%%%%%%%%%%%%%%%%%%%%%%%%%%%%%%
%%%%%%%%%%%%%%%%%%%%%%%%%%%%%%%%%%%%%%%%%%%%%%%%%%%%%%%%%%%
\section{Determinants and elementary divisors}
\label{sec:det}

This section is devoted to the proof of Theorem~\ref{thm:smith}.
\smallskip

We start by considering the matrix for a single $d$-gon, without
any diagonals, and denote it by $M_d$, thus
$$
M_d= 
\begin{pmatrix}
  0 & 1 & \ldots & \ldots & 1 \\
  1 & \ddots & \ddots & * & \vdots\\
 \vdots & \ddots &\ddots & \ddots & \vdots  \\
 \vdots & * &  \ddots   &  \ddots & 1 \\
 1 & \ldots & \ldots & 1 & 0
\end{pmatrix}.
$$
The determinants of matrices of the form $aM_d+bE_d$
($E_d$ the identity matrix) 
are wellknown (and easy to compute); we have
$$\det M_d = (-1)^{d-1} (d-1)\:.$$

For a sequence $a_1, \ldots, a_n$ we let $\Delta(a_1,\ldots,a_n)$
be the diagonal matrix with diagonal entries $a_1, \ldots, a_n$;
occasionally we collect multiple entries $a,\ldots,a$
(multiplicity $m$, say) and write
this in exponential form~$a^m$.
For a matrix $A$, we denote by $\cS(A)$ its Smith normal form,
i.e. the diagonal matrix with elementary divisors on the diagonal.
Background material and details on Smith normal forms and elementary 
divisors can be found in several textbooks, e.g. in Chapter 2 of
\cite{Newman}.

We now prove by induction that
$$\cS(M_d)= \Delta(1^{d-1},d-1)\, \text{for } d > 1\:.$$
For $d=2$, this is clearly true. In the induction step,
since $M_{d-1}$ has $1$ as an elementary divisor with multiplicity $d-2$,
the matrix
$M_d$ must have $1$ as an elementary divisor with multiplicity at least $d-2$.
Now the $(d-1)^{\text{st}}$ elementary divisor of $M_d$ divides the minor
$\det M_{d-1}=(-1)^{d-2}(d-2)$ of $M_d$
as well as $\det M_d=(-1)^{d-1}(d-1)$, hence we get one more elementary divisor~1,
and the final elementary divisor must be~$d-1$.
(Recall that the $i^{th}$ elementary divisor is the quotient of
the greatest common divisor of the minors of size $i$ and $i-1$,
respectively.) 
\smallskip

Thus we have the start $m=0$ of our induction, where we have $m+1$ $d$-gons in the $d$-angulation of the $n$-gon to be considered, i.e.,
$n=d+m(d-2)$.

\smallskip
Now assume we have already proved the result for $d$-angulations of
$\mathcal{P}_n$,
with $n=d+m(d-2)$,
and we want to prove the claim in the induction step
for a $d$-angulation of $\P_{n+d-2}$.

We know that there is a boundary
$d$-gon in the triangulation that is cut off from $\P_{n+d-2}$
by one diagonal; w.l.o.g. we may assume the diagonal to be between 
the vertices $1$ and~$n$ (in fact, relabelling the vertices of the 
polygon means
conjugation of the matrices $M_{\cT}$, and it is well known that 
conjugate matrices have the same determinant and the same elementary 
divisors). 
We have precise information on the relation to the remaining 
$n$-gon $\P_n$,
given for the respective matrices $M:=M_{\mathcal{T}}$ and $M'$ in Corollary~\ref{cor:recursion}.

The special structure of $M$ allows to transform this easily into a better form.
First, we subtract the sum of columns $1$ and $n$ from each column $n+1$ up to $n+d-2$.
This produces the zero matrix in the upper right corner, and transforms
the matrix $M_{d-2}$ sitting in the lower diagonal block into
$Z_{d-2}=-2E_{d-2}-M_{d-2}$.
Then subtracting the sum of rows $1$ and $n$ of the transformed matrix from
each row $n+1$ up to $n+d-2$ gives a block diagonal matrix with blocks $M'$ and $Z_{d-2}$.
Note that we have only used elementary row and column operations
that did not change either the determinant or the Smith normal form.

It is wellknown that $\det Z_{d-2} = (-1)^{d-2}(d-1)$; 
hence we have by induction
$$\det M = \det M' \cdot \det Z_{d-2} =
(-1)^{n-1+d-2} (d-1)^{m+2}
$$
as claimed.

Also the Smith normal form is easily determined.
By a similar induction argument as applied before for the matrix $M_d$, we have
$$\cS(Z_d)= \Delta(1^{d-1},d+1)\, \text{for } d \ge 1\:.$$
By induction, we have
$$\cS(M')= \Delta(1^{n-(m+1)},(d-1)^{m+1})\, \text{for } d > 1\:.$$
As $M$ is the block diagonal sum of $Z_{d-2}$ and $M'$, we get for $d>1$
$$\cS(M)= \Delta(1^{n-(m+1)+d-3},(d-1)^{m+2})= \Delta(1^{d+(m+1)(d-2)-(m+2)},(d-1)^{m+2})\:,$$
as desired, completing the proof of Theorem~\ref{thm:smith}.
\qed

%%%%%%%%%%%%%%%%%%%%%%%%%%%%%%%%%%%%%%%%%%%%%%%%%%%%%%%%%%%%
%%%%%%%%%%%%%%%%%%%%%%%%%%%%%%%%%%%%%%%%%%%%%%%%%%%%%%%%%%%%
\section{Frieze patterns from higher angulations}
\label{sec:friezes}

In this section we generalize Conway and Coxeter's frieze patterns
(which come from triangulations) to certain patterns of positive
integers (coming from $d$-angulations). It turns out that for this
generalized class of patterns we no longer
have the determinant $1$ condition for the $2 \times 2$ diamonds.
Instead, they all have determinant $0$ or $1$, and we will give
a combinatorial characterization of which of these
$2\times 2$-minors have determinant $0$; for triangulations
this latter case doesn't appear so that we also get as special case
a proof of the Conway-Coxeter result that triangulations
indeed give frieze patterns.

Exactly as for the case of triangulations one can
use the entries in the matrices $M_\cT$ to produce patterns of
integers. More precisely, take the upper (or lower) triangular
part of the matrix $M_\cT$ and use it as a fundamental region;
from this fundamental region the entire pattern is created by
applying successive glide reflections, as indicated in the following
picture.
%\begin{figure}
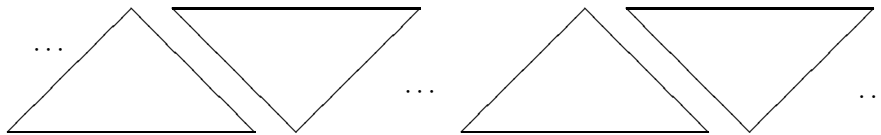
\[
  \xymatrix @-0.7pc @!0 {
    & & & & *{} \ar@{-}[dddrrr] & *{} \ar@{-}[dddrrr] \ar@{-}[rrrrrr] & & & & & &*{} & & & & *{} \ar@{-}[dddrrr] & *{} \ar@{-}[dddrrr] \ar@{-}[rrrrrr] & & & & & &*{}&*{}\\
    & & *{\cdots} & & & & & & & & & & & & & & & & & & & & *{} &\\
    & & & & & & & & & & & \cdots & & & & & & & & & &*{} & \ldots &\\
    & *{} \ar@{-}[uuurrr] \ar@{-}[rrrrrr] & & & & & & *{} & *{} \ar@{-}[uuurrr]& & & & *{} \ar@{-}[uuurrr] \ar@{-}[rrrrrr] & & & & & & *{} & *{} \ar@{-}[uuurrr]& *{}& & &\\
                        }
\]

Our aim is to generalize the defining determinant condition of Conway-Coxeter frieze
patterns. To this end we have to
consider the adjacent
$2\times 2$-minors in the matrices $M_\cT$ including the ones
between the last and first column (which correspond to the $2\times 2$-minors
appearing at the boundaries of the shifted fundamental regions).
Since the rows and columns of the matrices $M_\cT$ are indexed by
the vertices of the  $n$-gon $\mathcal{P}_n$
(in counterclockwise order), any such $2\times 2$-minor is
determined by a pair of boundary edges. For any such pair of boundary
edges $e=(i,i+1)$ and $f=(j,j+1)$ the corresponding minor has the form
$$d(e,f) := \det 
\begin{pmatrix} m_{i,j} & m_{i,j+1} \\ m_{i+1,j} & m_{i+1,j+1}
\end{pmatrix}
$$
where indices are to be reduced modulo $n$.

\begin{Theorem} \label{thm:hinge}
Let $\cT$ be a $d$-angulation of $\mathcal{P}_n$. 
Then the following holds for the $2\times 2$-minors of $M_\cT$ 
associated to boundary edges $e\neq f$ of $\mathcal{P}_n$:  
\begin{enumerate}
\item[{(a)}] $d(e,e)=-1$.
\item[{(b)}] $d(e,f)\in \{0,1\}$.
\item[{(c)}] $d(e,f)=1$ if and only if
there exists a sequence $e=z_0,z_1,\ldots,z_{s-1},z_{s}=f$, where $z_1, \ldots, z_{s-1}$ are diagonals in~$\cT$, such that the following holds for every $k\in\{0,1,\ldots,s-1\}$:
\begin{itemize}
\item[{(i)}] $z_k$ and $z_{k+1}$ belong to a common $d$-gon $p_k$
in $\mathcal{T}$;
\item[{(ii)}] the $d$-gons $p_k$ are pairwise different;
\item[{(iii)}] $z_{k}$ is incident to $z_{k+1}$.
\end{itemize}
\end{enumerate}
\end{Theorem}

\begin{Example}
Before entering the proof
let us illustrate the theorem with the
4-angulation $\cT$ of the 12-gon given in Example~\ref{ex:12gon}.
We consider the boundary edge $e=(4,5)$. Then we get the following
five sequences satisfying the
conditions in Theorem~\ref{thm:hinge}:
$e=(4,5),(5,6)=f$;
$e=(4,5),(4,7),(1,4),(3,4)=f$;
$e=(4,5),(4,7),(1,4),(1,2)=f$;
$e=(4,5),(4,7),(7,12),(7,8)=f$;
$e=(4,5),(4,7),(7,12),(11,12)=f$.

Note that, for instance, the sequence $e=(4,5),(4,7),(1,4),(1,12)$
is not allowed, since $(1,12)$ and $(1,4)$ belong to the same
quadrangle as $(4,7)$ and $(1,4)$ in the step before, i.e., 
condition (ii) of Theorem~\ref{thm:hinge} is violated.
By direct inspection of the matrix 
$\widetilde{M}_\cT=(\widetilde{m}_{ij})$
given in Example~\ref{ex:12gon} one observes that indeed the
non-zero $2\times 2$-minors appear for $j\in \{1,3,5,7,11\}$,
corresponding to the boundary edges $(j,j+1)$
which appear as destinations of the
above five sequences.
\end{Example}

The special case $d=3$, i.e., triangulations of polygons,
gives back the case of frieze patterns
in the sense of Conway and Coxeter.

\begin{Corollary}[\cite{CC1},\cite{CC2}]
Let $\cT$ be a triangulation of $\mathcal{P}_n$
and $M_\cT$ the corresponding matrix. Then all
adjacent $2\times 2$-minors are 1 and hence the above
construction produces a frieze pattern of integers.
\end{Corollary}

\begin{proof}
In a triangulation one can go from any boundary edge $e$
to any other boundary edge $f$ via a sequence of different
neighbouring triangles. The crucial point is that
for triangulations condition (iii) of Theorem~\ref{thm:hinge}
is empty. Hence there is always a suitable sequence from $e$ to
$f$ satisfying the conditions of Theorem~\ref{thm:hinge}.
\end{proof}

We now come to the proof of the main result of this section.

\begin{proof} (of Theorem~\ref{thm:hinge})
We use Theorem~\ref{thm:weights}
throughout the proof, i.e.,  that the numbers
$m_{i,j}=\widetilde{m}_{i,j}$ can be computed using the combinatorial
algorithm from Definition~\ref{def:weights}. 

(a) If $e=(i,i+1)$ then by Definition~\ref{def:weights}  we have
$$d(e,e) = \det \begin{pmatrix} m_{i,i} & m_{i,i+1} \\ m_{i+1,i} & m_{i+1,i+1}
\end{pmatrix}
= \det \begin{pmatrix} 0 & 1 \\ 1 & 0 \end{pmatrix} = -1.
$$
\smallskip

(b) and (c) Let $e=(i,i+1)$ and $f=(j,j+1)$ be different boundary edges
of $\mathcal{P}_n$.

{\it Case~1:} $e$ and $f$ belong to a common $d$-gon of $\mathcal{T}$.

Then, if there is a sequence as in the theorem, 
we must have $s=1$ since the $d$-gons have to be pairwise 
different by condition (ii); 
condition (iii) then forces $e$ and $f$ have a common endpoint.

Thus we have to show that $d(e,f)=1$
if $e$ and $f$ share an endpoint, and $d(e,f)=0$ otherwise. 

If $e$ and $f$ have a common endpoint, w.l.o.g.\ $j=i+1$, then
by Definition~\ref{def:weights} we get
$$d(e,f) = \det \begin{pmatrix} m_{i,j} & m_{i,j+1} \\ m_{i+1,j} & m_{i+1,j+1}
\end{pmatrix}
= \det \begin{pmatrix} 1 & 1 \\ 0 & 1 \end{pmatrix} = 1.
$$
Otherwise we have, again by Definition~\ref{def:weights}, that
$$d(e,f) = \det \begin{pmatrix} m_{i,j} & m_{i,j+1} \\ m_{i+1,j}
& m_{i+1,j+1} \end{pmatrix}
= \det \begin{pmatrix} 1 & 1 \\ 1 & 1 \end{pmatrix} = 0.
$$

{\em Case~2:} $e$ and $f$ do not belong to a common $d$-gon of
$\mathcal{T}$.

Recall that every $d$-angulation can be built by successively glueing
$d$-gons onto the boundary. By Definition~\ref{def:weights},
the computation of the numbers involved in the determinant $d(e,f)$
only uses those $d$-gons in the $d$-angulation $\mathcal{T}$ which
form the minimal convex subpolygon of $\mathcal{P}_n$
containing $e$ and $f$. Hence we may assume that $f=(j,j+1)$
belongs to a boundary $d$-gon $\alpha$
of $\mathcal{P}_n=\mathcal{P}'\cup \alpha$ which is cut off by the diagonal
$t$ with endpoints $u$ and $v$, as in the following figure
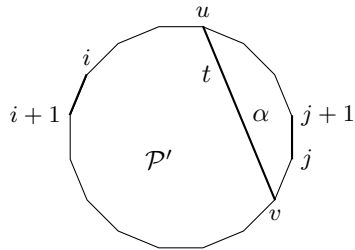
\[
  \begin{tikzpicture}[auto]
    \node[name=s, shape=regular polygon, regular polygon sides=16, minimum size=3cm, draw] {};
 \draw[thick] (s.corner 1) to (s.corner 12);
  \draw[thick] (s.corner 4) to (s.corner 5);
  \draw[thick] (s.corner 13) to (s.corner 14);
  \draw[shift=(s.corner 14)]  node[left=5pt]  {$\alpha$};
  \draw[shift=(s.corner 4)]  node[right=40pt]  {$t$};
  \draw[shift=(s.corner 6)]  node[right=25pt]  {$\mathcal{P}'$};
  \draw[shift=(s.corner 1)]  node[above]  {{\small $u$}};
  \draw[shift=(s.corner 12)]  node[below]  {{\small $v$}};
  \draw[shift=(s.corner 13)]  node[right]  {{\small $j$}};
  \draw[shift=(s.corner 14)]  node[right]  {{\small $j+1$}};
  \draw[shift=(s.corner 4)]  node[above]  {{\small $i$}};
  \draw[shift=(s.corner 5)]  node[left]  {{\small $i+1$}};
  \end{tikzpicture}
\]
If $f$ is not attached to $t$ then a sequence as in Theorem~\ref{thm:hinge} can not
exist for $e$ and $f$ in $\mathcal{P}_n$ since any such sequence
would have to involve $t$ as penultimate entry,
but then condition (iii) is not
satisfied. 
On the other hand, in this case 
$m_{i,j}=m_{i,j+1}$ and
$m_{i+1,j}=m_{i+1,j+1}$ and hence
$$d(e,f)= \det \begin{pmatrix} m_{i,j} & m_{i,j+1} \\ m_{i+1,j} & m_{i+1,j+1}
\end{pmatrix}
= \det \begin{pmatrix} m_{i,j} & m_{i,j} \\ m_{i+1,j} & m_{i+1,j}
\end{pmatrix} = 0.
$$

If $f$ is attached to $t$, say w.l.o.g.\ $j=v$, then we have
$$d(e,f) =
\det \begin{pmatrix} m_{i,v} & m_{i,v+1} \\ m_{i+1,v} & m_{i+1,v+1}
\end{pmatrix} =
\det \begin{pmatrix} m_{i,v} & m_{i,u}+m_{i,v} \\ m_{i+1,v} &
m_{i+1,u}+m_{i+1,v}
\end{pmatrix} =
\det \begin{pmatrix} m_{i,v} & m_{i,u} \\ m_{i+1,v} &
m_{i+1,u} \end{pmatrix} \:.
$$
This last determinant is just $d_{\mathcal{P}'}(e,t)$
where the index indicates that one considers the smaller
polygon $\mathcal{P}'$ in which $t$ is a boundary edge. 
Inductively, the value of $d_{\mathcal{P}'}(e,t)$ can only
be $0$ or $1$, hence at this point we have proven part (b) by induction
(with Case~1 as start of the induction).

Now a sequence from $e$ to $f$ as in Theorem~\ref{thm:hinge}
exists if and only if there exists such a
sequence for $e$ and~$t$ in the smaller polygon $\mathcal{P}'$; 
in fact this latter sequence then has to be extended by $f$ 
to give a sequence for $e$ and $f$ in $\mathcal{P}_n$. 
Inductively, this happens if and only if
$d_{\mathcal{P}'}(e,t)=1$. But as 
$d(e,f)=d_{\mathcal{P}'}(e,t)$, this implies that $d(e,f)=1$
if and only if a sequence as in Theorem~\ref{thm:hinge} exists
for $e$ and $f$ in $\mathcal{P}_n$.

This completes the proof of Theorem~\ref{thm:hinge}.
\end{proof}

%%%%%%%%%%%%%%%%%%%%%%%%%%%%%%%%%%%%%%%%%%%%%%%%%%%%%%%%%%%
%%%%%%%%%%%%%%%%%%%%%%%%%%%%%%%%%%%%%%%%%%%%%%%%%%%%%%%%%%%

%%%%%%%%%%%%%%%%%%%%%%%%%%%%%%%%%%%%%%%%%%%%%%%%%%%%%%%%%%%%%%%
%%%%%%%%%%%%%%%%%%%%%%%%%%%%%%%%%%%%%%%%%%%%%%%%%%%%%%%%%%%%%%%

\end{document}